\documentclass{article}
\pdfoutput=1
\usepackage[T1]{fontenc}
\usepackage[cp1250]{inputenc}
\usepackage{amsthm}
\usepackage{amsmath}
\usepackage{amsfonts}
\usepackage{multicol}
\usepackage{epsfig}
\usepackage{verbatim}
\usepackage[affil-it]{authblk}
\usepackage[authoryear]{natbib}
\usepackage{subcaption}
\usepackage{amssymb}
\usepackage{pdfpages}

\usepackage{hyperref}
\hypersetup{
	bookmarksnumbered,
	linkcolor=black,
	citecolor=black,
	colorlinks=true,
}%

\newtheorem{theorem}{Theorem}
\newtheorem{lemma}{Lemma}

\theoremstyle{definition}
\newtheorem{example}{Example}

\theoremstyle{plain}

\newtheorem{proposition}{Proposition}

\def\X{\mathfrak{X}}
\def\T{\mathfrak T}
\def\R{\mathbb R}

\def\C{\mathcal C}

\newcommand{\bb}{\mathbf b}
\newcommand{\cb}{\mathbf c}
\newcommand{\e}{\mathbf e}
\newcommand{\f}{\mathbf f}
\newcommand{\g}{\mathbf g}
\newcommand{\h}{\mathbf h}

\newcommand{\x}{\mathbf x}

\newcommand{\z}{\mathbf z}

\newcommand{\A}{\mathbf A}
\newcommand{\B}{\mathbf B}
\newcommand{\Cb}{\mathbf C}

\newcommand{\G}{\mathbf G}

\newcommand{\I}{\mathbf I}
\newcommand{\K}{\mathbf K}

\newcommand{\M}{\mathbf M}
\newcommand{\N}{\mathbf N}

\newcommand{\Q}{\mathbf Q}
\newcommand{\Sb}{\mathbf S}

\newcommand{\0}{\mathbf 0}
\newcommand{\1}{\mathbf 1}

\newcommand{\eig}{\gamma}
\def\diag{\mathrm{diag}}
\newcommand{\tr}{\mathrm{tr}}
\def\Var{\mathrm{Var}}

\title{Optimal experimental designs for treatment contrasts in heteroscedastic models with covariates}
\author{Samuel Rosa}
\date{\today} 

\begin{document}
	
\maketitle

\begin{abstract}
	In clinical trials, the response of a given subject often depends on the selected treatment as well as on some covariates. We study optimal approximate designs of experiments in the models with treatment and covariate effects. We allow for the variances of the responses to depend on the chosen treatments, which introduces heteroscedasticity into the models. For estimating systems of treatment contrasts and linear functions of the covariates, we extend known results on $D$-optimality of product designs by providing product designs that are optimal with respect to general eigenvalue-based criteria. In particular, $A$- and $E$-optimal product designs are obtained. We then formulate a method based on linear programming for constructing optimal designs with smaller supports from the optimal product designs. The sparser designs can be more easily converted to practically applicable exact designs. The provided results and the proposed sparsification method are demonstrated on some examples.
\end{abstract}

\section{Introduction}

In the experiments performed to estimate the effects of a selected set of treatments, it is common that the responses are also affected by some other experimental conditions (the covariates -- e.g., time trend, block effects, the ages or the gender of the subjects in a clinical trial). The effects of the covariates are traditionally considered to be nuisance effects in the experimental design literature (e.g., \cite{Cox51}, \cite{MajumdarNotz}, \cite{AtkinsonDonev}, \cite{JacrouxMajumdar}, \cite{RosaHarman16}). Recently, \cite{Atkinson15} suggested that particularly with the growing importance of personalized medicine, the covariate effects may also be of prominent interest.

Moreover, the interest often lies in some functions of the model parameters rather than in the parameters themselves. The comparisons of test treatments with a control or with the placebo is a common example, especially in clinical trials. The design and analysis are sometimes further complicated by the presence of heteroscedasticity -- the responses under different treatments may have different variances. In the present paper, we study optimal approximate designs in heteroscedastic models with treatment effects and covariate effects where there is interest in a set of treatment contrasts and in a set of linear combinations of the covariate effects.

For the abovementioned settings, \cite{Atkinson15} studied $D$-optimality in a model with only two treatments. \cite{WangAi} extended his results by providing $D$-optimal product designs for arbitrary numbers of treatments. The $D$-optimality is the most popular criterion in optimal design literature and is particularly nice to work with analytically. However, if the interest lies in a set of treatment contrasts (e.g., in the test treatment-control comparisons, which are a natural choice when placebo is included in the clinical trial), $D$-optimality tends to ignore the special interest in the chosen contrasts; i.e., this criterion tends to select designs that do not provide more information on the contrasts of interest compared to other treatment contrasts. As such, $D$-optimality is generally not recommended for such experimental interests (cf. \cite{HedayatEA}, \cite{MorganWang10}). This also corresponds to the observation by \cite{WangAi} that if the experimental objective is to estimate the test treatment-control comparisons and all covariate effects, then the $D$-optimal designs are the same as if there was interest in all covariate effects and a uniform interest in all the treatments. That is, in such a case, the $D$-optimality does not place any special emphasis on the test treatment-control comparisons. 

In contrast, the $A$-optimality criterion possesses a natural statistical interpretation for the considered settings -- it minimizes the average variance for the linear functions of interest. Hence, $A$-optimality is very popular for the comparisons with the control (e.g., see \cite{HedayatEA}). Recently, $E$-optimality was also argued to be meaningful for estimating treatment contrasts (e.g., \cite{MorganWang}, \cite{Rosa18block}). In this paper, we therefore extend the results by \cite{WangAi} to other eigenvalue-based optimality criteria (see Section \ref{sOptProd}).  In particular, we obtain $\Phi_p$-optimal product designs for the other Kiefer's $\Phi_p$-optimality criteria besides $D$-optimality, including $A$- and $E$-optimality.

The observation that the product designs are generally optimal for multi-factor models (like the treatment-covariate one) is extensively used in the literature; e.g., see \cite{SchwabeWierich}, \cite{Schwabe}, \cite{RodriguezOrtiz}, \cite{GrasshoffEA}. One drawback of the optimal product designs is that they have large supports, and therefore it is sometimes difficult to construct designs for the actual experiments from the product designs by rounding procedures\footnote{The product designs, like other approximate designs, specify only proportions of trials to be performed for the particular experimental conditions, as will be formalized later. To obtain actual integer numbers of trials, some rounding is generally performed.} (e.g., the well-known efficient rounding procedure by \cite{PukelsheimRieder92}). However, in Section \ref{sNonProd}, we provide an entire class of $\Phi$-optimal designs characterized by linear constraints. This allows us to formulate a linear programming method for constructing optimal designs with smaller supports from the optimal product designs. Similar approach was employed by \cite{RosaHarman16} in a homoscedastic model, where the covariate effects were considered to be nuisance parameters.

The application of the theoretical results, and in particular the construction of the optimal designs with sparser supports and their usefulness in obtaining efficient exact designs of experiments is demonstrated on some examples in Section \ref{sExamples}.

\subsection{Notation}

By $\1_n$ and $\0_n$, we denote the vector of ones and the vector of zeros in $\R^n$, respectively. The vector $\e_i$ has 1 on the $i$th position and zeros elsewhere. The identity matrix is denoted by $\I_n$ and the $m \times n$ matrix of zeros is denoted by $\0_{m \times n}$. For brevity, we sometimes omit the subscripts expressing the dimensions. The expression $\diag(a_1,\ldots, a_n)$, where $a_i\in\R$, denotes the diagonal matrix with $a_1, \ldots, a_n$ on diagonal. If $\A$ and $\B$ are matrices, $\diag(\A,\B)$ denotes the corresponding block diagonal matrix. The smallest eigenvalue of a nonnegative definite matrix $\A$ is denoted by $\eig_{\min}(\A)$. Given a matrix $\A$, the symbol $\A^-$ denotes a generalized inverse of $\A$, and $\A^+$ is the Moore-Penrose pseudoinverse of $\A$.

\subsection{The model}

Consider the model
\begin{equation}\label{eModel1}
	y(i,k)=\tau_{i} + \mu + \g^T(k) \beta + \varepsilon(i,k),
\end{equation}
where $i \in \{1,\ldots,v_1\}$ is the chosen treatment and $k \in \T$, $\vert \T \vert = d$, represents the chosen covariates. The vector $\tau=(\tau_1,\ldots,\tau_{v_1})$ represents the treatment effects, $\mu$ is the constant term, $\beta = (\beta_1, \ldots, \beta_{v_2})^T$ are the covariate effects, and $\g(k) \in \R^{v_2}$ is the regression function for the covariates $k$. The errors $\varepsilon$ are uncorrelated with zero mean, and their variance depends on the treatment chosen for the $i$th trial: $\Var(\varepsilon(i,k))=\sigma^2/\lambda_{i}$. The positive function $\lambda: \{1,\ldots,v_1\} \to \R_{++}$ is called the efficiency function and is assumed to be known.

We assume that the set $\T$ of all potential covariates is finite, which is often the case in practice: the covariates are either naturally discrete (e.g., blocks) or the continuous covariates are discretized. For ease of notation, we number the covariates: $\T=\{1,\ldots,d\}$. Model \eqref{eModel1} can be expressed in a compact form as $y(i,k) = \f^T(i,k)\theta + \varepsilon(i,k)$, where $\theta=(\tau^T,\mu,\beta^T)^T$, and $\f(i,k) = (\e_i^T,1,\g^T(k))^T$ is the regression function for the given $i$ and $k$.

The approximate design $\xi$ (or the design $\xi$, in short) is a probability measure on the design space $\X = \{1,\ldots,v_1\} \otimes \{1,\ldots,d\}$; i.e., $\xi$ is a nonnegative function from $\X$ to $\R$ that satisfies $\sum_{i,k} \xi(i,k) = 1$. The value $\xi(i,k)$ represents the proportion of all trials that are performed with treatment $i$ and covariates $k$. The exact design $\xi_e$, which describes an actual experiment consisting of $n$ trials, specifies the number of trials $\xi_e(i,k)$ that are performed with treatment $i$ and covariates $k$ for each $i$ and $k$. The limit on the number of trials means that $\sum_{i,k} \xi_e(i,k) = n$. In this paper, we consider approximate designs, unless specified otherwise.

The moment matrix of a design $\xi$ is $\M(\xi) = \sum_{i,k} \xi(i,k)\lambda_i \f(i,k)\f^T(i,k)$. Let $\h(k) := (1,\g^T(k))^T$; then, $\M(\xi)$ can be expressed in the block form
$$
\M(\xi) = \begin{bmatrix}
\M_{11}(\xi) & \M_{12}(\xi) \\ \M_{12}^T(\xi) & \M_{22}(\xi)
\end{bmatrix},
$$
where $\M_{11}(\xi)=\diag(\lambda_1 w_1, \ldots, \lambda_{v_1} w_{v_1})$, $\M_{12}(\xi)=\big(\sum_k \lambda_1\xi(1,k)\h(k), \ldots, \sum_k \lambda_{v_1}\xi(v_1,k)\h(k)\big)^T$
and $\M_{22}(\xi) = \sum_{i,k} \lambda_i \xi(i,k) \h(k) \h^T(k)$.

The experimental interest in a set of $s_1$ treatment contrasts $\Q_1^T\tau$ and in a set of $s_2$ linear functions of the covariate effects $\K^T\beta$, where $\Q_1\in\R^{v_1 \times s_1}$ and $\K \in \R^{v_2 \times s_2}$, can be expressed as $\A^T\theta$, where
$$\A=\begin{bmatrix}
\Q_1 & \0_{v_1 \times s_2}  \\ \0_{(v_2+1) \times s_1} & \Q_2
\end{bmatrix}=\diag(\Q_1,\Q_2)$$
and $\Q_2^T = (\0_{s_2},\K^T)$. Because $\Q_1^T\tau$ is a system of contrasts, the matrix $\Q_1$ satisfies $\Q^T\1_{v_1} = \0_{s_1}$. The subsystem $\A^T\theta$ is estimable under $\xi$ if $\C(\A) \subseteq \C(\M(\xi))$. In such a case, we say that $\xi$ is feasible for $\A^T\theta$. The constant term $\mu$ is not estimable because of the presence of the treatment effects; therefore, the first row of matrix $\Q_2$ is a row of zeros, as formulated above.
Let both $\Q_1$ and $\K$ be of full column rank. We also suppose that there is interest in all treatments; i.e., no row of $\Q_1$ is a row of zeros. The information matrix for $\A^T\theta$ of a feasible $\xi$ is $\N_\A(\xi) = (\A^T \M^-(\xi) \A)^{-1}$.

A design $\xi^*$ is $\Phi$-optimal if it maximizes $\Phi(\N_\A(\xi))$ for a given functional $\Phi$ of the information matrix. For example, $\xi^*$ is $D$-optimal if it maximizes $\det(\N_\A(\xi))$, $A$-optimal if it maximizes $1/\tr(\N_\A^{-1}(\xi))$ and $E$-optimal if it maximizes $\eig_{\min}(\N_\A(\xi))$. The mentioned optimality criteria can be extended to an entire class of the so-called Kiefer's $\Phi_p$-optimality criteria, $p \in [-\infty,0]$ (see \cite{puk}, Chapter 6):
$$
\Phi_p(\N) = \begin{cases}
(\det(\N))^{1/s}, & p=0; \\
(\frac{1}{s} \tr(\N^p))^{1/p}, & p\in(-\infty,0); \\
\eig_{\min}(\N), & p=-\infty,
\end{cases}
$$
where $\N$ is an $s \times s$ positive definite matrix. The criteria of $D$-, $A$- and $E$-optimality are obtained by setting $p=0$, $-1$ and $-\infty$, respectively.

We require that the optimality criteria possess some basic properties so that they correctly measure the amount of the obtained information. In particular, the criteria must be positively homogeneous, concave, nonnegative, nonconstant and upper semicontinuous; such criteria are called \emph{information functions} (\cite{puk}, Chapter 5). The optimality criteria are usually eigenvalue based, i.e., they depend only on the eigenvalues of the information matrix. In the present paper, we restrict ourselves to eigenvalue-based information functions. One common class of such functions is the class of the $\Phi_p$-criteria.

If the system of interest is rank deficient (i.e., if $\A$ is not of full column rank), then $\A^T\M^-(\xi)\A$ is never non-singular, and therefore $(\A^T\M^-(\xi)\A)^{-1}$ does not exist. In such cases, especially for eigenvalue-based optimality criteria, the information matrix for $\A^T\theta$ is set to be $\N_\A(\xi) = (\A^T\M^-(\xi)\A)^{+}$; see Section 8.18 by \cite{puk}. The $\Phi_p$-optimality criteria are then defined on the \emph{positive} eigenvalues of $\N_\A(\xi)$. The system $\tau_i - \bar{\tau}$ ($i=1,\ldots,v_1$) for estimating centered treatment effects, represented by $\Q_1 = \I_{v_1} - \1_{v_1}\1_{v_1}^T/v_1$, is a simple example of a rank-deficient system of treatment contrasts.

\subsection{Marginal models}

For the first marginal model of \eqref{eModel1}, we consider the heteroscedastic model for treatment effects
\begin{equation}\label{eModelTreat}
	y(i) = \tau_{i} + \varepsilon(i),
\end{equation}
where $\Var(\varepsilon(i))=\sigma^2/\lambda_{i}$. We denote the \emph{marginal treatment design} $w$ as the approximate design in \eqref{eModelTreat}. That is, $w$ specifies $v_1$ nonnegative treatment weights that sum to one; we denote these weights as $w_1, \ldots, w_{v_1}$. The moment matrix of $w$ is $\M_1(w) = \diag(\lambda_1 w_1, \ldots, \lambda_{v_1} w_{v_1})$, and $w$ is feasible for $\Q_1^T\tau$ if $w_i>0$ for all $i=1,\ldots,v_1$, denoted as $w>0$. The information matrix for $\Q_1^T\tau$ of $w>0$ is $\N_{\Q_1}(w) = (\Q_1^T \M_1^{-1}(w) \Q_1)^{-1}$.
\bigskip

For the second marginal model, we consider the homoscedastic model for covariates with constant term:
\begin{equation}\label{eModelNuis}
	y(k) = \mu + \g^T(k) \beta + \varepsilon_k.
\end{equation}
Then, the \emph{marginal covariate design} $\alpha$ is an approximate design in \eqref{eModelNuis}. The marginal covariate design is analogously characterized by the $d$ covariate weights $\alpha_1, \ldots, \alpha_{d}$, and its moment matrix is 
$$
\M_2(\alpha) = \sum_{k=1}^{d}\alpha_k \h(k)\h^T(k) = \begin{bmatrix}
1 & \sum_k \alpha_k \g^T(k) \\
\sum_k\alpha_k \g(k) & \sum_k\alpha_k \g(k) \g^T(k)
\end{bmatrix}.
$$
The subsystem $\K^T\beta$ is estimable if $\C(\Q_2) \subseteq \C(\M_2(\alpha))$, which is equivalent to $\C(\K) \subseteq \Sb(\alpha)$, where $\Sb(\alpha) = \sum_k \alpha_k\g(k)\g^T(k) - (\sum_k \alpha_k \g(k))(\sum_k \alpha_k \g^T(k))$ is the Schur complement of $\M_2(\alpha)$.
The Schur complement of a nonnegative definite matrix in the block form
\begin{equation}\label{ePartitioned}
\B=\begin{bmatrix}
\B_{11} & \B_{12} \\ \B_{12}^T & \B_{22}
\end{bmatrix}
\end{equation}
is $\B_\Sb = \B_{22} - \B_{12}^T \B_{11}^{-1} \B_{12}$. More precisely, $\B_\Sb$ is the Schur complement of $\B_{11}$ in $\B$.

The information matrix for $\K^T\beta$ of a feasible $\alpha$ is $\N_\K(\alpha) = (\Q_2^T \M_2^-(\alpha)\Q_2)^{-1}$. The information matrix can be expressed using the Schur complement, because the matrix
\begin{equation}\label{eGinv}
	\G = \begin{bmatrix}
	\B_{11}^{-1} + \B_{11}^{-1}\B_{12}\B_{\Sb}^-\B_{12}^T \B_{11}^{-1} & - \B_{11}^{-1}\B_{12}\B_{\Sb}^- \\
	-\B_{\Sb}^-\B_{12}^T \B_{11}^{-1} & \B_{\Sb}^-
	\end{bmatrix},
\end{equation}
is a generalized matrix of an arbitrary nonnegative definite matrix $\B$ given by \eqref{ePartitioned}; see Theorem 9.6.1 by \cite{Harville}. Formula \eqref{eGinv} implies for $\M_2(\alpha)$ that $\N_\K(\alpha) = (\K^T \Sb^-(\alpha)\K)^{-1}$.
\bigskip

For a design $\xi$ in model \eqref{eModel1}, the marginal treatment design $w$ of $\xi$ is given by $w_i = \sum_k \xi(i,k)$ ($i=1,\ldots,v_1$) and the marginal covariate design $\alpha$ of $\xi$ is given by $\alpha_k = \sum_i \xi(i,k)$ ($k=1,\ldots,d$).

\section{Optimal product designs}
\label{sOptProd}

The product design $\xi = w \otimes \alpha$ of the marginal designs $w$ and $\alpha$ satisfies $\xi(i,k) = w_i\alpha_k$ ($i=1, \ldots, v_1$, $k=1,\ldots,d$).

\begin{proposition}\label{pProd}
	If $w>0$ and if $\alpha$ is feasible for $\K^T\beta$, then $\xi=w\otimes \alpha$ is feasible for $\A^T\theta$ and its information matrix is
	\begin{equation}\label{eInfMatProd}
	\N_\A(w\otimes\alpha)=\begin{bmatrix}
	\N_{\Q_1}(w) & 0 \\ 0 &  (\sum_i \lambda_i w_i) \N_\K(\alpha)
	\end{bmatrix}.
	\end{equation}
\end{proposition}

\begin{proof}
	For $\xi=w \otimes \alpha$, a generalized inverse of $\M(\xi)$ can be expressed as in \eqref{eGinv} using the Schur complement 
	$$\begin{aligned}
	\Sb_1(\xi)
	&= \M_{22}(\xi) - \M_{12}^T(\xi) \M_{11}^{-1}(\xi) \M_{12}(\xi) \\
	&= \Big(\sum_{i=1}^{v_1} \lambda_i w_i\Big)\Big(\sum_{k=1}^{d} \alpha_k \h(k)\h^T(k) - (\sum_{k=1}^{d} \alpha_k \h(k))(\sum_{k=1}^{d} \alpha_k \h^T(k))\Big) \\
	&= \Big(\sum_{i=1}^{v_1} \lambda_i w_i\Big) \begin{bmatrix}
	0 & \0^T \\ \0 & \Sb(\alpha)
	\end{bmatrix}.
	\end{aligned}$$
	Observe that $\M_{11}(\xi) = \M_1(w)$ and that
	$$\begin{aligned}
	\M_{11}^{-1}(\xi)\M_{12}(\xi) 
	&= \diag((\lambda_1 w_1)^{-1}, \ldots, (\lambda_{v_1} w_{v_1})^{-1})\begin{bmatrix}
	\lambda_1 w_1 \\ \vdots \\ \lambda_{v_1} w_{v_1}
	\end{bmatrix}\sum_{k=1}^d\alpha_k \h^T(k)  \\
	&= \1_{v_1} \sum_{k=1}^d\alpha_k \h^T(k).
	\end{aligned}$$
	Because $\Q_1^T\1 = \0 $, we obtain $\Q_1^T \M_{11}^{-1}(\xi)\M_{12}(\xi) = \0$.
	Let us calculate $\M(\xi)\M^-(\xi)\A$, where $\M^-(\xi)$ is given by \eqref{eGinv}. By employing the abovementioned observations, it is straightforward to show that
	$\M(\xi)\M^-(\xi)\A = \diag(\Q_1, \Sb_1(\xi)\Sb_1^-(\xi)\Q_2)$. Because $\alpha$ is feasible for $\K^T\beta$, we have $\C(\K) \subseteq \C(\Sb(\alpha))$, which implies $\C(\Q_2) \subseteq \C(\Sb_1(\xi))$. Then, $\Sb_1(\xi)\Sb_1^-(\xi)\Q_2 = \Q_2$, which yields $\M(\xi)\M^-(\xi)\A = \A$; i.e., $\xi$ is feasible for $\A^T\theta$.
	
	The information matrix of $\xi$ can be expressed using the same generalized inverse $\M^-(\xi)$ and the fact that $\M_{11}^{-1}(\xi)\M_{12}(\xi)\Q_1 = \0$:
	$$
	\N_\A^{-1}(\xi) = \begin{bmatrix}
	\Q_1^T \M_{1}^{-1}(w)\Q_1 & \0 \\ \0 &  \Q_2^T \Sb_1^-(\xi) \Q_2
	\end{bmatrix},
	$$
	which is equivalent to \eqref{eInfMatProd}.
\end{proof}

The following preliminary lemma shows that no design is better with respect to any eigenvalue-based criterion than the product of its marginals.

\begin{lemma}\label{lProd}
	Let $\Phi$ be an eigenvalue-based information function, let $\xi$ be a design in model \eqref{eModel1} and let $w$ and $\alpha$ be its marginal designs. Then,
	$\Phi(\N_\A(\xi)) \leq \Phi(\N_\A(w \otimes \alpha))$.
\end{lemma}
\begin{proof}
	Similarly to the proof of Lemma 3.1 by \cite{Schwabe}, let us consider the modification of model \eqref{eModel1}, where $\tau_i$ is changed to $-\tau_{i}$. This corresponds to the change in regressors from $\f(i,k)=(\e_i^T,1,\g^T(k))$ to $\tilde{\f}(i,k)=(-\e_i^T,1,\g^T(k))$. The moment matrix $\tilde{\M}(\xi)$ in the modified model is
	$$\tilde{\M}(\xi)=\begin{bmatrix}
	\M_{11}(\xi) & - \M_{12}(\xi) \\ -\M_{12}^T(\xi) & \M_{22}(\xi)
	\end{bmatrix}.$$
	Then, the information matrix changes from
	$\N_\A(\xi)$
	to
	$$\tilde{\N}_\A(\xi)=\begin{bmatrix}
	-\I & \0 \\ \0 & \I
	\end{bmatrix} \N_\A(\xi) \begin{bmatrix}
	-\I & \0 \\ \0 & \I
	\end{bmatrix}.$$
	Because the matrix $\diag(-\I,\I)$ is orthogonal, the information matrices $\N_\A(\xi)$ and $\tilde{\N}_\A(\xi)$ have the same eigenvalues; hence $\Phi(\N_\A(\xi)) = \Phi(\tilde{\N}_\A(\xi))$. From concavity of $\Phi$ it follows that 
	$$\Phi(\frac{1}{2}(\N_\A(\xi) + \tilde{\N}_\A(\xi))) \geq \frac{1}{2}\Phi(\N_\A(\xi)) + \frac{1}{2}\Phi(\tilde{\N}_\A(\xi)) = \Phi(\N_\A(\xi)).$$
	Moreover, $(\N_\A(\xi) + \tilde{\N}_\A(\xi))/2 = \diag(\N_{11}(\xi),\N_{22}(\xi))=:\N_*(\xi)$. By expressing the generalized inverse $\M^-(\xi)$ through the Schur complement $\Sb_1(\xi)$ (see \eqref{eGinv}), we obtain that 
	$\N_*(\xi) = (\A^T\diag(\M_{11}^{-1} + \M_{11}^{-1}\M_{12} \Sb_1^- \M_{12}^T \M_{11}^{-1}, \Sb_1^-)\A)^{-1}$.
	Moreover, $\diag(\M_{11}^{-1} + \M_{11}^{-1}\M_{12} \Sb_1^- \M_{12}^T \M_{11}^{-1}, \Sb_1^-) \succeq \diag(\M_{11}^{-1}, \Sb_1^-)$, which implies that $\N_*(\xi) \preceq (\A^T\diag(\M_{11}^{-1},\Sb_1^-)\A)^{-1} = \N_\A(w\otimes\alpha)$. Hence, $\Phi(\N_\A(w\otimes\alpha)) \geq \Phi(\N_*(\xi)) \geq \Phi(\N_\A(\xi))$.
\end{proof}

For any eigenvalue-based $\Phi$, Lemma \ref{lProd} implies that if there exists a $\Phi$-optimal design, then there exists a product design that is $\Phi$-optimal among all designs. In particular, if $\xi^*$ is $\Phi$-optimal with marginal treatment design $w^*$ and marginal covariate design $\alpha^*$, then $w^* \otimes \alpha^*$ is $\Phi$-optimal.

\begin{theorem}\label{tOptProd}
	Let $\Phi$ be an eigenvalue-based information function. Then, if there exists a $\Phi$-optimal design for $\A^T\theta$ in \eqref{eModel1}, then there exists a product design in \eqref{eModel1} that is $\Phi$-optimal for $\A^T\theta$.
\end{theorem}

Theorem \ref{tOptProd} shows that to obtain a $\Phi$-optimal design it suffices to find a $\Phi$-optimal product design. For $D$-optimality, the problem of finding a $\Phi$-optimal product design simplifies because of its multiplicative form: 
$$\det(\N_\A(w\otimes\alpha)) = (\sum_{i=1}^{v_1} \lambda_i w_i)^{s_2}\det(\N_{\Q_1}(w)) \det(\N_\K(\alpha)).$$ This means that the optimal $w$ and $\alpha$ can be computed separately, which was extensively used by \cite{WangAi}.

For the other $\Phi_p$-optimality criteria, we reduce the complexity of the problem as follows: For $p \in (-\infty,0)$, we have 
\begin{equation}\label{ePhip}
\Phi_p(\N_\A(w\otimes\alpha)) = \Big(s^{-1}\big( \tr(\N_{\Q_1}^p(w)) + (\sum_i \lambda_i w_i)^p  \tr(\N_{\K}^p(\alpha))\big)\Big)^{1/p}.
\end{equation}
The additive form allows one to first calculate the covariate design $\alpha^*$, which is $\Phi_p$-optimal in the marginal model \eqref{eModelNuis},  by maximizing $\Phi_p(\N_\K(\alpha))$ and then calculate the corresponding optimal $w^*$ by maximizing $\Phi_p(\N_\A(w\otimes\alpha^*))$ over all $w$. Then, from \eqref{ePhip} it follows that the product design $w^*\otimes \alpha^*$ is $\Phi_p$-optimal. Analogously, to calculate the $E$-optimal ($p=-\infty$) product design, one can first find the $E$-optimal $\alpha^*$ in \eqref{eModelNuis}, which maximizes $\eig_{\min}(\N_\K(\alpha))$, and then the optimal $w^*$ that maximizes $\eig_{\min}(\N_\A(w\otimes\alpha^*))$. Therefore, for any $\Phi_p$-optimality criterion, the optimization problem of size $v_1 \cdot d$ can be split into two maximization problems of sizes $d$ and $v_1$, respectively, even for $p$ other than $0$.

\begin{theorem}\label{tPhipOpt}
	(i) Let $p \in (-\infty,0)$, let $\alpha^*$ be $\Phi_p$-optimal for $\K^T\beta$ in \eqref{eModelNuis} and denote $\varphi^* = \tr(\N^p_\K(\alpha^*))$. Let $w^*$ maximize $\Phi_p(\N_\A(w\otimes \alpha^*)) = (s^{-1}(\tr(\N^p_{\Q_1}(w)) +  (\sum_i \lambda_i w_i)^p\varphi^*))^{1/p}$. Then, $\xi^* = w^* \otimes \alpha^*$ is $\Phi$-optimal for $\A^T\theta$ in \eqref{eModel1}.
	
	(ii) Let $\alpha^*$ be $E$-optimal for $\K^T\beta$ in \eqref{eModelNuis} and let $w^*$ maximize $\eig_{\min}(\N_\A(w\otimes \alpha^*))$. Then, $\xi^* = w^* \otimes \alpha^*$ is $E$-optimal for $\A^T\theta$ in \eqref{eModel1}.
\end{theorem}

Theorem \ref{tPhipOpt} shows that $\alpha^*$ in the $\Phi_p$-optimal product design $w^* \otimes \alpha^*$ is the $\Phi_p$-optimal marginal covariate design, but $w^*$ is generally \emph{not} $\Phi_p$-optimal in \eqref{eModelTreat}. Due to heteroscedasticity, a slight correction (represented by $\varphi^*$) needs to be present to calculate the ``optimal'' marginal treatment design $w^*$.

\section{Optimal non-product designs}
\label{sNonProd}

In the previous section, we showed that to calculate an optimal design, it suffices to restrict oneself to the product designs. However, in model \eqref{eModel1}, besides an optimal product design, there usually exists a rich class of optimal designs that are not of the product form. Fortunately, once a $\Phi$-optimal product design $\xi^*$ is found, other $\Phi$-optimal designs can generally be constructed from $\xi^*$. The following theorem provides such a construction.

\begin{theorem}\label{tOptNonProd}
	Let $\Phi$ be an eigenvalue-based information function and let $\xi^*=w^*\otimes\alpha^*$ be the $\Phi$-optimal product design for $\A^T\theta$. Let $\xi$ satisfy
	$\M(\xi) \G\A = \A$,
	where 
	$$\G=\diag(\M_1^{-1}(w^*),(\sum_i \lambda_i w_i^*)^{-1}\M_2^-(\alpha^*))$$
	and where $\M_2^-(\alpha^*)$ is any generalized inverse of $\M_2(\alpha^*)$. Then, $\xi$ is $\Phi$-optimal for $\A^T\theta$.
	
	In particular, let $\xi$ satisfy
	\begin{equation}\label{eOptTreat}
		w_i = w_i^*, \quad i=1,\ldots,v_1
	\end{equation}
	\begin{equation}\label{eCovRes}
	\left[
	\frac{1}{w_1^*}\sum_{k=1}^{d} \xi(1,k)\g(k),\, \ldots,\, \frac{1}{w_v^*}\sum_{k=1}^{d} \xi(v_1,k)\g(k)
	\right] \Q_1 = \0_{v_2 \times s_1},
	\end{equation}
	\begin{equation}\label{eTreatRes}
		\begin{bmatrix}
		\sum_k (\xi(1,k) - w_1^*\alpha_k^*)\g^T(k) \\
		\vdots \\
		\sum_k (\xi(v_1,k) - w_{v_1}^*\alpha_k^*)\g^T(k)
		\end{bmatrix} \Sb^-(\alpha^*)\K = \0_{v_1 \times s_2},
	\end{equation}
	\begin{equation}\label{eOptCov}
		\Big(\sum_{i=1}^{v_1} \lambda_i w_i^*\Big)^{-1}\Big((\sum_{i,k} \lambda_i\xi(i,k)\g(k)\g^T(k)) - (\sum_{i,k}\lambda_i \xi(i,k)\g(k))\sum_{k=1}^{v_2}\alpha^*_k \g^T(k)\Big) \Sb^-(\alpha^*)\K = \K,
	\end{equation}
	where $w$ is the marginal treatment design of $\xi$. Then, $\xi$ is $\Phi$-optimal for $\A^T\theta$.
\end{theorem}

\begin{proof}
	Lemma 1 by \cite{RosaHarman16} shows that if $\tilde{\M}$ is a non-negative definite matrix and if a design $\xi$ satisfies $\M(\xi)\tilde{\M}^-\A=\A$, then $\xi$ is feasible for $\A^T\theta$ and $\A^T\M^-(\xi)\A = \A^T\tilde{\M}^-\A$. We choose $\tilde{\M} = \diag(\M_1(w^*), (\sum_i \lambda_i w_i^*)\M_2(\alpha^*))$ and 
	$$\G=\diag(\M_1^{-1}(w^*), (\sum_i \lambda_i w_i^*)^{-1}\M_2^-(\alpha^*)).$$ 
	Then $\G$ is a generalized inverse of $\tilde{\M}$, and any $\xi$ satisfying $\M(\xi)\G\A = \A$ satisfies  $\A^T\M^-(\xi)\A = \A^T\G\A = \N_\A^{-1}(\xi^*)$; i.e., such $\xi$ is $\Phi$-optimal.
	
	For the second part, for the generalized inverse of $\M_2(\alpha^*)$ we choose
	\begin{equation}\label{eGinvM2}
		\M_2^-(\alpha^*) =
		\begin{bmatrix}
		1 + (\sum_k \alpha^*_k \g^T(k))\Sb^-(\alpha^*) (\sum_k \alpha^*_k \g^T(k)) & -\sum_k \alpha^*_k \g^T(k) \Sb^-(\alpha^*) \\
		-\Sb^-(\alpha^*)\sum_k \alpha^*_k \g(k) & \Sb^-(\alpha^*)
		\end{bmatrix}.
	\end{equation}
	The condition $\M(\xi)\G\A=\A$ consists of the following equalities:
	\begin{equation}\label{eIM1}
		\M_{11}(\xi)\M_{1}^{-1}(w^*) \Q_1 = \Q_1
	\end{equation}
	\begin{equation}\label{eIM2}
		\M_{12}(\xi) \M_{2}^-(\alpha^*)\Q_2 = \0
	\end{equation}
	\begin{equation}\label{eIM3}
		\M_{12}^T(\xi)\M_{1}^{-1}(w^*) \Q_1 = \0
	\end{equation}
	\begin{equation}\label{eIM4}
		(\sum_{i=1}^{v_1}\lambda_i w_i^*)^{-1}\M_{22}(\xi) \M_{2}^-(\alpha^*)\Q_2 = \Q_2
	\end{equation}
	Because no row of $\Q_1$ is a row of zeros, condition \eqref{eIM1} implies that $w_i = w_i^*$. Equality \eqref{eIM2} can be expressed as
	$$
		\Big(\begin{bmatrix}
		\sum_k \lambda_1\xi(1,k)\g^T(k) \\
		\vdots \\
		\sum_k \lambda_v\xi(v,k)\g^T(k)
		\end{bmatrix} -
		\begin{bmatrix}
		\lambda_1 w_1^*  \\
		\vdots \\
		\lambda_v w_v^* 
		\end{bmatrix} \sum_{k=1}^d \alpha^*_k \g^T(k) \Big) \Sb^-(\alpha^*)\K = \0,
	$$
	which can be simplified to \eqref{eTreatRes}.
	Equality \eqref{eIM3} can be simplified to
	$$
	\left[
	\frac{1}{w_1^*}\sum_{k=1}^{v_2} \xi(1,k)\g(k),\, \ldots,\, \frac{1}{w_v^*}\sum_{k=1}^{v_2} \xi(v,k)\g(k)
	\right] \Q_1 = \0.
	$$
	Finally, condition \eqref{eIM4} can be expressed by the following equalities:
	\begin{equation}\label{eIM4a}
	\Big( (\sum_{i,k} \lambda_i \xi(i,k)\g^T(k)) - (\sum_{i=1}^v \lambda_i w_i^*)\sum_{k=1}^{v_2} \alpha_k^* \g^T(k)\Big) \Sb^-(\alpha^*)\K = \0^T,
	\end{equation}
	$$
	\Big(\sum_{i=1}^{v_1} \lambda_i w_i^*\Big)^{-1}\Big((\sum_{i,k} \lambda_i \xi(i,k)\g(k)\g^T(k)) - (\sum_{i,k}\lambda_i\xi(i,k)g(k))\sum_{k=1}^{v_2}\alpha^*_k \g^T(k)\Big) \Sb^-(\alpha^*)\K = \K.
	$$
	However, \eqref{eIM4a} can be simplified to
	$$
	\Big( \sum_{i=1}^{v_1} \lambda_i \sum_{k=1}^{v_2} (\xi(i,k) - w_i^*\alpha_k^*)\g^T(k)\Big) \Sb^-(\alpha^*)\K = \0^T,
	$$
	which follows from \eqref{eTreatRes}.
\end{proof}

Because the conditions in Theorem \ref{tOptNonProd} are linear, optimal designs satisfying these conditions can be calculated via linear programming (LP) once $w^*$ and $\alpha^*$ are computed:
\begin{equation}\label{eLP}
	\min\{ \cb^T\x \ \vert \
	\x \in \R^{v_1 d}, \Cb\x = \bb,
	\x \geq \0 \},
\end{equation}
where $\x$ represents the $v_1 d$ design values, $\Cb\x = \bb$ consists of the conditions given by Theorem \ref{tOptNonProd} and of the design constraint $\sum_j \x_j = 1$. The vector $\cb$ can be chosen arbitrarily, as the objective is to find any $\x\geq\0$ that satisfies $\Cb\x=\bb$.

The possibility of employing linear programming has two crucial advantages. The more obvious one is that it is relatively simple to solve the LP problems, and most mathematical software packages (e.g., MATLAB and R) contain reliable and fast LP solvers. The other advantage is that the vertices of the set of the feasible solutions of \eqref{eLP} have high numbers of zeroes among all feasible solutions; for technical details, see, e.g., Theorem 2.4 by \cite{BertsimasTsitsiklis}. In fact, any vertex solution of \eqref{eLP} is guaranteed to have at least $v_1 d - r$ zeros, where $r = \mathrm{rank}(\Cb)$. This is beneficial, because such designs with small supports can then be obtained by solving the LP problem via the simplex method, as this method provides vertex solutions. Note that the MATLAB implementation of the interior point method also seems to provide vertex solutions. Theorem \ref{tOptNonProd} therefore allows us to formulate a linear programming ``sparsification'' method of the product designs based on solving \eqref{eLP}. In Section \ref{sExamples}, we demonstrate the applicability of this method.

\paragraph{Remarks}
\begin{itemize}
	\item For simpler settings (e.g., in the homoscedastic case or if the interest lies in $\Q_1^T\tau$ only), the obtained results simplify correspondingly. For instance, if the interest lies only in a set of treatment contrasts (i.e., $\A = (\Q_1^T,\0)^T$), then the product design $w^* \otimes \alpha$ is $\Phi$-optimal for any marginal covariate design $\alpha$, where $w^*$ is a $\Phi$-optimal marginal treatment design for $\Q_1^T\tau$. Moreover, for such $\A^T\theta$, any design $\xi$ that satisfies \eqref{eCovRes} and whose marginal treatment design is $w^*$ is $\Phi$-optimal. The designs satisfying \eqref{eCovRes} were denoted as resistant to nuisance effects in a slightly different context by \cite{RosaHarman16}. In the present settings, designs satisfying \eqref{eOptTreat} and \eqref{eCovRes} may be called covariate resistant: if the interest lies only in $\Q_1^T\tau$, then no relevant information is lost under such designs due to the presence of covariates.
	
	\item In the case of a rank-deficient system, both the results of Section \ref{sOptProd} and of Section \ref{sNonProd} hold, as the proofs in Section \ref{sOptProd} can be easily adapted to $\N_\A(\xi) = (\A^T\M^-(\xi)\A)^+$ and Theorem \ref{tOptNonProd} ensures that $\A^T\M^-(\xi)\A = \A^T\tilde{\M}^-\A$, which also implies the equality of the rank-deficient information matrices.
	
	\item If a $\Phi$-optimal design $\xi^*$ is known, other $\Phi$-optimal designs can trivially be found by solving the linear equality $\sum_{i,k} \xi(i,k)\f(i,k)\f^T(i,k) = \M(\xi^*)$ that guarantees that the design $\xi$ has the same moment matrix as $\xi^*$. The conditions in Theorem \ref{tOptNonProd} are more general, as it can be shown that any design satisfying $\M(\xi) = \M(\xi^*)$ also satisfies $\M(\xi)\G\A=\A$ with $\G$ given by Theorem \ref{tOptNonProd}. Conditions \eqref{eOptTreat}-\eqref{eCovRes} can generally be satisfied also by designs $\xi$ that do not have the same moment matrix as $\xi^*$; only the information matrices of $\xi$ and $\xi^*$ are guaranteed to coincide.
\end{itemize}

\section{Examples}
\label{sExamples}

\begin{example}\label{exContCov}
	Consider a model with effects of $v_2$ continuous covariates
	$$
	y(i,\z) = \tau_{i} + \mu + \z^T\beta + \varepsilon,
	$$
	where $\z \in [-1,1]^{v_2}$, as in Example 2 by \cite{WangAi}. By a discretization of the continuous covariates, say $z_i \in \{\pm j/10 \ \vert \ j=0,\ldots,10\}$, the model becomes a special case of \eqref{eModel1}. Suppose that we wish to obtain $A$-optimal designs for estimating the comparisons with the control $\tau_i - \tau_1$, $i=2,\ldots,v_1$ and all the covariate effects; i.e., $\Q=(-\1_{v_1-1},\I_{v_1-1})^T$ and $\K=\I_{v_2}$. Then, the marginal covariate design $\alpha^*$ that is uniform on $\{-1,1\}^{v_2}$ is $A$-optimal for $\K^T\beta$. The corresponding information matrix is $\N_\K(\alpha^*) = \I_{v_2}$, and hence $\varphi^* = \tr(\I_{v_2}^{-1}) = v_2$. Then, the optimal marginal treatment design $w^*$ can be obtained by minimizing $\tr(\Q^T\M_1^{-1}(w)\Q) + v_2(\sum_i \lambda_i w_i)^{-1}$.
	
	Let $v_1 = v_2 = 3$ and let $\lambda_1 = 9$, $\lambda_2 = \lambda_3 = 1$; i.e., the observations under the control treatment have smaller variances. Then, the vector of the optimal treatment weights given by $w^*$ is $(0.236, 0.382, 0.382)^T$, and hence the product design $\xi^* = w^* \otimes \alpha^*$ (see Table \ref{tblOptProd}) is $A$-optimal for $\A^T\theta$. As observed in Theorem \ref{tPhipOpt}, the optimal treatment design $w^*$ depends on $\varphi^*$, which is equal to the number of covariates $v_2$ in the current example. Figure \ref{fEx1} depicts the dependence of $w_1^*$ on $v_2$ for the abovementioned settings.
	
	\begin{table}[t]
		\centering
		\begin{tabular}{l|llllllll}
			\hline\noalign{\smallskip}
			$i\backslash k$ & 1 & 2 & 3 & 4 & 5 & 6 & 7 & 8 \\
			\noalign{\smallskip}\hline\noalign{\smallskip}
			1 & 0.0295  &  0.0295  &  0.0295  &  0.0295  &  0.0295  &  0.0295  &  0.0295  &  0.0295 \\
			2 & 0.0477  &  0.0477  &  0.0477  &  0.0477  &  0.0477  &  0.0477  &  0.0477  &  0.0477 \\
			3 & 0.0477  &  0.0477  &  0.0477  &  0.0477  &  0.0477  &  0.0477  &  0.0477  &  0.0477 \\
			\noalign{\smallskip}\hline
		\end{tabular}
		\caption{Optimal product design for Example \ref{exContCov}. The indices $k=1,\ldots,8$ denote the active covariate values $\z = (-1,-1,-1)^T, (-1,-1,1)^T, \ldots, (1,1,1)^T$ arranged in the lexicographical order; the treatments are denoted by $i=1,2,3$.}
		\label{tblOptProd}
	\end{table}
	\begin{table}[t]
		\centering
		\begin{tabular}{l|llllllll}
			\hline\noalign{\smallskip}
			$i\backslash k$ & 1 & 2 & 3 & 4 & 5 & 6 & 7 & 8 \\
			\noalign{\smallskip}\hline\noalign{\smallskip}
			1  &  0.0378  &  0  &  0.0212  &  0.0591  &  0.0212  &  0.0591  &  0.0378  &  0 \\
			2  &  0  &  0.1909  &  0  &  0  &  0  &  0  &  0.1909  &  0 \\
			3  &  0.1909  &  0  &  0  &  0  &  0  &  0  &  0  &  0.1909 \\
			\noalign{\smallskip}\hline
		\end{tabular}
		\caption{Optimal non-product design for Example \ref{exContCov}; the notation is the same as in Table \ref{tblOptProd}.}
		\label{tblOptNonProd}
	\end{table}
	\begin{table}[t]
		\centering
		\begin{tabular}{l|llllllll}
			\hline\noalign{\smallskip}
			$i\backslash k$ & 1 & 2 & 3 & 4 & 5 & 6 & 7 & 8 \\
			\noalign{\smallskip}\hline\noalign{\smallskip}
			1  &   2  &   0  &   1  &   3  &   1  &   3  &   2  &   0 \\
			2  &   0  &   9  &   0  &   0  &   0  &   0  &   9  &   0 \\
			3  &   9  &   0  &   0  &   0  &   0  &   0  &   0  &   9 \\
			\noalign{\smallskip}\hline
		\end{tabular}
		\caption{Exact design for $n=48$ trials constructed from the optimal non-product design given in Table \ref{tblOptNonProd}.}
		\label{tblOptNonProdR}
	\end{table}
	\begin{figure}[t]
		\centering
		\includegraphics[width=0.5\textwidth]{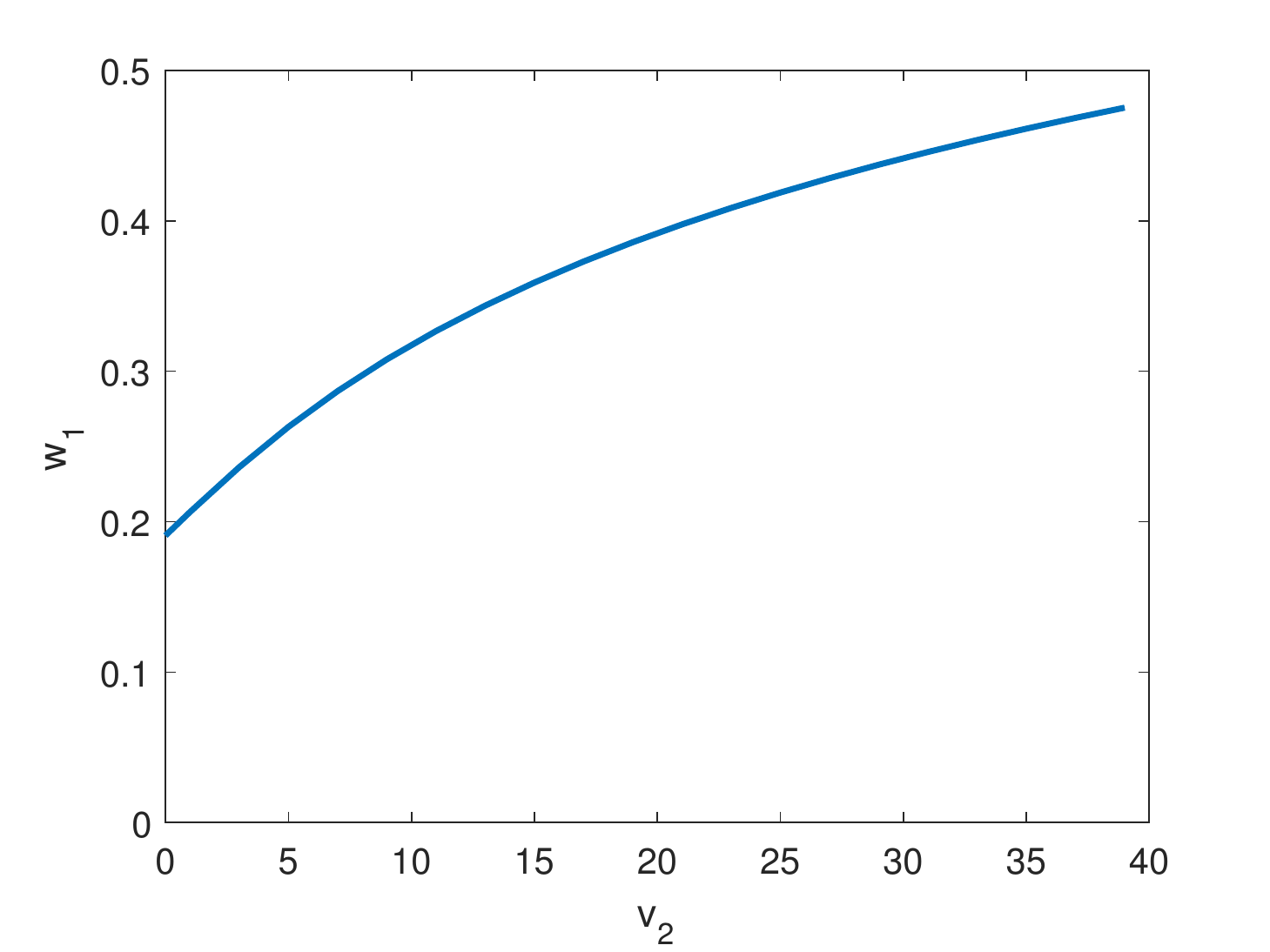}
		\caption{The dependence of $w_1^*$ for the $A$-optimal product design $w^* \otimes \alpha^*$ on the number of covariates $v_2$ in Example \ref{exContCov}, where $\Q = (-\1, \I)^T$, $\K=\I$, $v_1 = 3$, $\lambda_1=9$ and $\lambda_2 = \lambda_3 = 1$.}
		\label{fEx1}
	\end{figure}
	
	An $A$-optimal design, say $\xi_s^*$, that is supported on a smaller number of design points was computed by solving the linear program \eqref{eLP}. Although $\xi_s^*$ was allowed to attain non-zero values even outside of the support of $\xi^*$, it turns out that the support of $\xi_s^*$ is a subset of the support of $\xi^*$. Therefore, $\xi_s^*$ can also be expressed by considering only $\z \in \{-1,1\}^3$, see Table \ref{tblOptNonProd}. Whereas the support of the optimal product design is of size 24, the support of the optimal non-product design is only of size 10. One can also observe that the marginal treatment design of $\xi_s^*$ is $w^*$, but $\alpha^*$ is not its marginal covariate design.
	
	Let us demonstrate the usefulness of the sparsely supported designs for constructing the exact designs by the efficient rounding procedure by \cite{PukelsheimRieder92}. For $n=48$ trials the rounded product design $\xi_{ep}$ satisfies $\xi_{ep}(i,\z) = 2$ for each $i$ and for each $\z \in \{-1,1\}^3$; the rounded non-product design $\xi_{es}$ is given in Table \ref{tblOptNonProdR}. To compare the quality of the obtained designs, we calculate their efficiencies: $\mathrm{eff}(\xi_e) = \Phi(\N_\A(\xi_e/n))/\Phi(\N_\A(\xi^*))$, where $\xi_e$ is an exact design for $n$ trials and $\xi^*$ is the $\Phi$-optimal approximate design. We have $\mathrm{eff}(\xi_{ep}) = 0.9641$ for the rounded product design and $\mathrm{eff}(\xi_{es}) = 0.9991$ in the non-product case. The very high efficiency of the exact design $\xi_{es}$ relative to the optimal approximate design means that $\xi_{es}$ is either an optimal exact design, or very nearly optimal.
	
	The efficient rounding method requires that the number of trials $n$ be equal or greater than the support size of the approximate design. Therefore, the efficient rounding of $\xi^*$ cannot be performed for $n \leq 23$, as the support size of $\xi^*$ is $24$. However, the rounding of the non-product $\xi_s^*$ can be done even for $10 \leq n \leq 23$.  This is demonstrated in Figure \ref{fEx1b}, which shows efficiencies of the respective rounded designs for varying $n$. Note that the rounded non-product design is not always more efficient than the rounded product design.
\end{example}
\begin{figure}[t]
	\centering
	\includegraphics[width=0.5\textwidth]{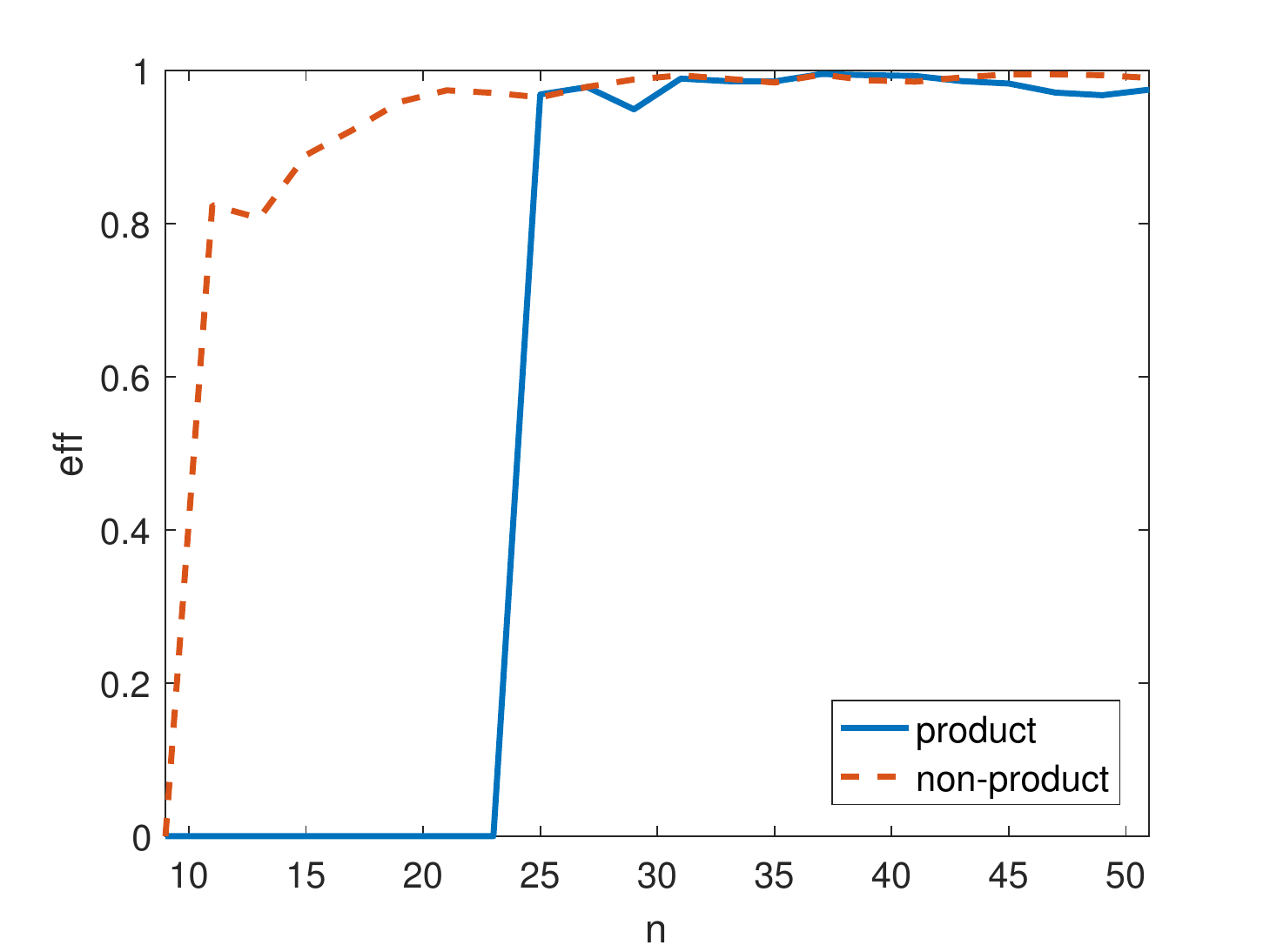}
	\caption{$A$-efficiencies of the rounded product and non-product designs in Example \ref{exContCov} for varying number of observations $n$. Zero efficiency means that for the corresponding $n$, the efficient rounding cannot be applied for the design.}
	\label{fEx1b}
\end{figure}

\begin{example}\label{exRowColumn}
	Suppose that the responses are affected, besides the treatment effects, by two qualitative covariates:
	$$
	y(i,j,k) = \tau_i + \mu + \eta_j + \rho_k + \varepsilon,
	$$
	where $j \in\{1,\ldots,r\}$ is the level of the first covariate, $k\in\{1,\ldots,c\}$ is the level of the second covariate, and $\eta_j$ and $\rho_k$ are the respective covariate effects. In the optimal design literature, such a model is usually known as two-way elimination of heterogeneity (e.g., see \cite{HedayatEA}) or a row-column model (e.g., see \cite{Jacroux}), where the trials are split into $r$ rows and $c$ columns. If the effects of one of the covariates is disregarded, we obtain the well-known model with block effects. Suppose that all the treatments are of the same interest, which can be expressed by estimating the treatment effects corrected for the mean $\tau_i - \bar{\tau}$; i.e., $\Q_1 = \I_{v_1} - \1_{v_1}\1_{v_1}^T/{v_1}$. Similarly, let $\K = \diag(\I_r - \1_r\1_r^T/r, \I_c - \1_c\1_c^T/c)$. Let $v_1 = 3$, $r = 3$, $c = 5$ and $\lambda = (4,1,1)^T$. We shall provide $E$-optimal designs for estimating $\A^T\theta$.
	
	
	The $E$-optimal marginal covariate design for $\K^T\beta$ is uniform: $\alpha^*(j,k) = 1/15$ for all $j,k$. The optimal criterion value $\Phi_E(\N_{\Q_2}(\alpha^*))$ is $0.2$, and then the corresponding optimal $w^*$ calculated using Theorem \ref{tPhipOpt} is given by the vector of treatment weights $(0.273, 0.364, 0.364)^T$. Then, the $E$-optimal product design $w^* \otimes \alpha^*$ is supported on all the $v_1 \times rc = 45$ design points, with values $\xi(1,j,k) = 0.0182$ and $\xi(2,j,k) = \xi(3,j,k) = 0.0242$ for each $j$ and $k$. By solving the linear program \eqref{eLP}, an $E$-optimal design $\xi_s^*$ with a smaller number of support points can be obtained. Such a design is given in Table \ref{tblOptNonProd2}. The non-product design $\xi_s^*$ has a smaller support of only size 28, as expressed by the large number of zeroes in Table \ref{tblOptNonProd2}. Table \ref{tblOptNonProd2R} gives the rounded non-product design $\xi_s^*$ for $n=40$ trials; for this number of trials, the efficient rounding of the optimal product design cannot be performed because of its large support size of $48>n$. The efficiency of the exact design obtained by the rounding of $\xi_s^*$ is $0.8493$.
\end{example}

\begin{table}[t]
	\centering
	\begin{tabular}{l|lllll}
		\hline\noalign{\smallskip}
		$\xi(1,j,k)$ & 1 & 2 & 3 & 4 & 5 \\
		\noalign{\smallskip}\hline\noalign{\smallskip}
		1  &  0.0303  &  0.0121  &  0.0303  &  0.0121  &  0.0061 \\
		2  &  0.0061  &  0.0303  &  0.0121  &  0.0242  &  0.0182 \\
		3  &  0.0182  &  0.0121  &  0.0121  &  0.0182  &  0.0303 \\
		\hline\noalign{\smallskip}
		$\xi(2,j,k)$ & 1 & 2 & 3 & 4 & 5 \\
		\noalign{\smallskip}\hline\noalign{\smallskip}
        1 & 0       &  0       &  0  &  0.0485 &   0.0727 \\
        2 & 0.0242  &       0  &  0.0727  &  0.0242    &     0 \\
        3 & 0.0485  &  0.0727  &       0   &      0     &    0 \\
		\hline\noalign{\smallskip}
		$\xi(3,j,k)$ & 1 & 2 & 3 & 4 & 5 \\
		\noalign{\smallskip}\hline\noalign{\smallskip}
		1    &     0  &  0.0727    &     0  &  0.0242  &  0.0242 \\
		2    &     0.0727   &      0     &    0   &      0  &  0.0485 \\
		3    &     0      &   0  &  0.0727  &  0.0485    &     0 \\
		\noalign{\smallskip}\hline
	\end{tabular}
	\caption{Optimal non-product design for Example \ref{exRowColumn}. The table contains the values $\xi(i,j,k)$, where $i$ is the treatment, $j$ is the first covariate level (expressed in rows) and $k$ is the second covariate level (expressed in columns). For example, the value $0.0121$ in row 2 and column 3 of the first ``block'' of the table means that $\xi(1,2,3) = 0.0121$.}
	\label{tblOptNonProd2}
\end{table}
\begin{table}[t]
	\centering
	\begin{tabular}{l|lllll}
		\hline\noalign{\smallskip}
		$\xi(1,j,k)$ & 1 & 2 & 3 & 4 & 5 \\
		\noalign{\smallskip}\hline\noalign{\smallskip}
		1   &  1  &   1  &   1  &   1   &  1 \\
		2   &  1  &   1  &   1  &   1   &  1 \\
		3   &  1  &   1  &   1  &   1   &  1 \\
		\hline\noalign{\smallskip}
		$\xi(2,j,k)$ & 1 & 2 & 3 & 4 & 5 \\
		\noalign{\smallskip}\hline\noalign{\smallskip}
		1  &   0   &  0   &  0   &  2  &   2 \\
		2  &   1   &  0   &  2   &  1  &   0 \\
		3  &   2   &  2   &  0   &  0  &   0 \\
		\hline\noalign{\smallskip}
		$\xi(3,j,k)$ & 1 & 2 & 3 & 4 & 5 \\
		\noalign{\smallskip}\hline\noalign{\smallskip}
		1  &   0  &   2  &   0  &   1  &   1 \\
		2  &   3  &   0  &   0  &   0  &   2 \\
		3  &   0  &   0  &   2  &   2  &   0 \\
		\noalign{\smallskip}\hline
	\end{tabular}
	\caption{Exact design for $n=40$ trials constructed from the optimal non-product design given in Table \ref{tblOptNonProd2}.}
	\label{tblOptNonProd2R}
\end{table}

\begin{example}\label{exTrendRes}
	Let us demonstrate the obtained theoretical results in settings, where there is no interest in the covariates. Consider a model with exponential trend effect, which is considered to be a nuisance:
	$$y(i,k) = \tau_i + \mu + g(k)\beta + \varepsilon,$$
	where $g(k) = e^k / \sum_j e^j$, $k=1,\ldots,d$. In each time $k = 1,\ldots,d$, the same number of trials must be performed, which results in the design constraint $\sum_i \xi(i,k) = 1/d$ for each $k$. Suppose that the interest lies in the treatment-control comparisons only; i.e., $\Q^T=(-\1_{v_1-1}, \I_{v_1-1})$ and $\A^T=(\Q^T,\0)$. Let $v_1=4$, $d=6$, $\lambda=(1,1,2,3)^T$ and consider the $A$-optimality criterion.
	
	The $A$-optimal marginal treatment design $w^*$ can be found by maximizing $\Phi_A(\N_{\Q}(w))$, and its values are given by the vector $(0.431, 0.249, 0.176, 0.144)^T$. The $A$-optimal product design is then $\xi^* = w^* \otimes \alpha$, where the marginal covariate design $\alpha_k = 1/d$ for $k=1,\ldots,d$ is implied by the design constraints $\sum_i \xi(i,k) = 1/d$. An optimal non-product design can be obtained as in Theorem \ref{tOptNonProd} by solving the linear program with constraints $\xi(i,k) \geq 0$ for all $i,k$, $\sum_i \xi(i,k) = 1/d$ for all $k$, and $\M(\xi)\G\A = \A$, where $\G=\diag(\M_1^{-1}(w^*),\0)$. The resulting design $\xi_s^*$, which is supported on 12 design points, is given in Table \ref{tblOptNonProd3}; the optimal product design is supported on all the 24 points.
	
	Because of the requirement that exactly one trial should be performed in each time moment, the usual rounding methods cannot be used to obtain efficient exact designs. A natural rounding method for the sparsely supported non-product designs is to select in each time the treatment with the highest design value; such a rounded design based on $\xi_s^*$ has efficiency 0.8871 and is given in Table \ref{tblOptNonProd3R}. Note that it is unclear how the optimal product designs can be rounded in the current example -- these designs do not provide any information on the suggested time sequence of the treatments, only the optimal treatment proportions are obtained. For instance, the aforementioned method would result for $\xi^*$ in a singular design that selects the first treatment in each time moment.
\end{example}

\begin{table}[t]
	\centering
	\begin{tabular}{l|llllll}
		\hline\noalign{\smallskip}
		$i \backslash k$ & 1 & 2 & 3 & 4 & 5 & 6 \\
		\noalign{\smallskip}\hline\noalign{\smallskip}
		1  &  0.0880  &  0.1667 &   0.0552  &  0.0166   &      0  &  0.1048 \\
		2  &  0.0787  &     0     &    0      &   0  &  0.1667  &  0.0036 \\
		3  &  0    &     0    &     0   & 0.1501   &      0  &  0.0260 \\
		4  &  0    &     0   & 0.1115    &     0     &    0  &  0.0323 \\
		\noalign{\smallskip}\hline
	\end{tabular}
	\caption{Optimal non-product design for Example \ref{exTrendRes}. $i$ -- treatment, $k$ -- time moment.}
	\label{tblOptNonProd3}
\end{table}
\begin{table}[t]
	\centering
	\begin{tabular}{l|llllll}
		\hline\noalign{\smallskip}
		$i \backslash k$ & 1 & 2 & 3 & 4 & 5 & 6 \\
		\noalign{\smallskip}\hline\noalign{\smallskip}
		1   &  1   &  1  &   0  &   0  &   0  &   1 \\
		2   &  0   &  0  &   0  &   0  &   1  &   0 \\
		3   &  0   &  0  &   0  &   1  &   0  &   0 \\
		4   &  0   &  0  &   1  &   0  &   0  &   0 \\
		\noalign{\smallskip}\hline
	\end{tabular}
	\caption{Exact design constructed from the optimal non-product design given in Table \ref{tblOptNonProd3}.}
	\label{tblOptNonProd3R}
\end{table}

\section{Discussion}

Although $D$-optimality is very beneficial analytically and computationally in the considered model, we showed that other eigenvalue-based optimality criteria like $A$- and $E$-optimality are not much more difficult to work with (Theorems \ref{tOptProd} and \ref{tPhipOpt}). Because the latter place actual emphasis on the selected system of interest, they seem to be more appropriate in the current settings.

The use of optimal product designs is also computationally very beneficial, because they allow one to reduce the $v_1 d$-dimensional optimal design problem in the multi-factor model \eqref{eModel1} to two simpler problems of sizes $v_1$ and $d$ for the two marginal models. However, the product designs suffer from large support sizes, as noted earlier. The support sizes of the product designs can in some settings be reduced by combinatorial approaches; e.g., by replacing full factorials by orthogonal arrays as in \cite{GrasshoffEA04}. Such reductions of support sizes are however very dependent on the particular well-studied models with a very regular structure that can be made use of in the combinatorial arguments.

To reduce the support sizes in more general settings, we provide an entire class of optimal designs (see Theorem \ref{tOptNonProd}) characterized by linear constraints. Such a characterization allows one to use linear programming for constructing optimal non-product designs. That the obtained constraints for optimal designs are linear is particularly useful, because the linear programming routines tend to provide vertex solutions, which have high numbers of zeros. As a result, the solvers ``automatically'' provide designs with small supports. As we demonstrated on examples in Section \ref{sExamples}, it seems that the method tends to allow for significant reductions of the support sizes. The use of linear programming also means that the algorithm is reliable and fast. Unlike the combinatorial approaches, the proposed method is not tailored for specific settings -- it can be used for any model of the form \eqref{eModel1}.  

We would like to emphasize that the proposed support reduction method is not particularly tied to the current model \eqref{eModel1}. As noted earlier, the method was applied for a similar model by \cite{RosaHarman16}. Moreover, the proposed approach can be used for any linear regression model where an optimal design $\xi^*$ with a large support is found, either analytically or by design algorithms. Then one can find designs that satisfy either $\M(\xi) = \M(\xi^*)$ or $\M(\xi)\G\A = \A$ with a suitably chosen $\G$ as in Theorem \ref{tOptNonProd} via linear programming to obtain vertex solutions. As shown in the present paper, the vertex solutions correspond to designs that tend to have smaller supports and are therefore generally more useful for practical purposes.

\bibliographystyle{plainnat}
\bibliography{rosa}

\end{document}